\documentclass[12pt,a4paper]{amsart}

\usepackage{amssymb}
\usepackage{amsrefs}
\usepackage[all]{xy}
\usepackage{hyperref} 

\title{On moduli spaces for quasi-tilted algebras}

\author{Grzegorz Bobi\'nski}

\address{Faculty of Mathematics and Computer Science \\ Nicolaus Copernicus University \\ ul.~Chopina 12/18 \\ 87-100 Toru\'n \\ Poland}

\email{gregbob@mat.umk.pl}

\newtheorem*{conjecture}{Conjecture}

\newcounter{claim}[section]

\newtheorem{corollary}[claim]{Corollary}
\newtheorem{lemma}[claim]{Lemma}
\newtheorem{proposition}[claim]{Proposition}

\newcommand{\bc}{\mathbf{c}}
\newcommand{\bd}{\mathbf{d}}
\newcommand{\be}{\mathbf{e}}
\newcommand{\bh}{\mathbf{h}}
\newcommand{\bm}{\mathbf{m}}
\newcommand{\bx}{\mathbf{x}}
\newcommand{\by}{\mathbf{y}}
\newcommand{\bone}{\mathbf{1}}
\newcommand{\blambda}{\mbox{\boldmath $\lambda$}}

\newcommand{\calC}{\mathcal{C}}
\newcommand{\calF}{\mathcal{F}}
\newcommand{\calL}{\mathcal{L}}
\newcommand{\calM}{\mathcal{M}}
\newcommand{\calO}{\mathcal{O}}
\newcommand{\calP}{\mathcal{P}}
\newcommand{\calR}{\mathcal{R}}
\newcommand{\calT}{\mathcal{T}}
\newcommand{\calU}{\mathcal{U}}
\newcommand{\calV}{\mathcal{V}}
\newcommand{\calW}{\mathcal{W}}
\newcommand{\calX}{\mathcal{X}}
\newcommand{\calY}{\mathcal{Y}}
\newcommand{\calZ}{\mathcal{Z}}

\newcommand{\bbM}{\mathbb{M}}
\newcommand{\bbN}{\mathbb{N}}
\newcommand{\bbP}{\mathbb{P}}
\newcommand{\bbZ}{\mathbb{Z}}

\let\mod=\undefined
\DeclareMathOperator{\Du}{D} %
\DeclareMathOperator{\GL}{GL} %
\DeclareMathOperator{\op}{op} %
\DeclareMathOperator{\SI}{SI} %
\DeclareMathOperator{\Ann}{Ann} %
\DeclareMathOperator{\End}{End} %
\DeclareMathOperator{\Ext}{Ext} %
\DeclareMathOperator{\Hom}{Hom} %
\DeclareMathOperator{\mod}{mod} %
\DeclareMathOperator{\rep}{rep} %
\DeclareMathOperator{\sstab}{ss} %
\DeclareMathOperator{\chr}{char} %
\DeclareMathOperator{\spn}{span} %
\DeclareMathOperator{\pdim}{pdim} %
\DeclareMathOperator{\Proj}{Proj} %
\DeclareMathOperator{\supp}{supp} %
\DeclareMathOperator{\Coker}{Coker} %
\DeclareMathOperator{\gldim}{gldim} %
\DeclareMathOperator{\bdim}{\mathbf{dim}} %

\newcommand{\ol}{\overline}



\begin{document}

\begin{abstract}
We prove that if a quasi-tilted algebra is tame, then the associated moduli spaces are products of projective spaces. Together with an earlier result of Chindris this gives a geometric characterization of the tame quasi-tilted algebras. In proof we use knowledge of the representation theory of the tame quasi-tilted algebras and a construction of semi-invariants as determinants.
\end{abstract}

\maketitle

Throughout the article $\Bbbk$ is an algebraically closed field of characteristic $0$.

There is a well-known dichotomy for finite dimensional algebras due to Drozd~\cite{Drozd}: every algebra is either tame or wild, but not both. Here a finite dimensional algebra is called tame if, for each dimension $d$, the indecomposable $d$-dimensional modules form finitely many one-parameter families. On the other hand, an algebra $\Lambda$ is wild if the classification of $\Lambda$-modules is as difficult as the classification of the pairs of two (non-commuting) endomorphisms of a finite dimensional vector space (the latter problem is considered to be hopeless).

The above definitions of tame and wild algebras are of geometric nature. This encourages people to look for characterizations of representation type, which use properties of geometric objects associated to them, like for example modules varieties (some results of this type can be found in~\cites{BobinskiSkowronski, SkowronskiZwara}). In particular, Skowro\'nski and Weyman~\cite{SkowronskiWeyman} have proved that a hereditary algebra $\Lambda$ is tame if and only if all the corresponding rings of semi-invariants are complete intersections. Inspired by this result Chindris has initiated a programme, whose aim is to characterize representation type in terms of (rational) invariant theory~\cites{Chindris2009, Chindris2011, Chindris2013}. As a result of his studies he has formulated together with Carroll~\cite{CarrollChindris}, the following conjecture, which they have attributed to Weyman:

\begin{conjecture}
Let $\Lambda$ be an algebra. Then the following conditions are equivalent:
\begin{enumerate}

\item
$\Lambda$ is of tame representation type.

\item
For any dimension vector $\bd$, for any irreducible component $\calC$ of $\mod_\Lambda (\bd)$, and for any weight $\theta$ such that $\calC_\theta^{\sstab} \neq \varnothing$, $\calM (\calC)_\theta^{\sstab}$ is a product of projective spaces.

\end{enumerate}
\end{conjecture}

This conjecture has no chances to hold is such generality. Obvious counterexamples are local wild algebras. There is also a counterexample due to Ringel of a triangular (no cycles in the Gabriel quiver) wild algebra, such that all the associated moduli space are points.

The aim of this paper is to verify this conjecture for the quasi-tilted algebras. The implication $(2) \implies (1)$ has been proved for the tilted algebras by Chindris~\cite{Chindris2013}*{Propostion~4.1}. In fact, as it has been explained to me by Chindris, his proof of the implication $(2) \implies (1)$ generalizes to the quasi-tilted algebras. More precisely, in~\cite{Chindris2013} Chindris uses a result of Kerner~\cite{Kerner} stating that a wild tilted algebras has a convex subalgebra which is wild concealed. In the case of the quasi-tilted algebras one has to use results of Lenzing and Skowro\'nski~\cite{LenzingSkowronski} (every wild quasi-tilted algebra has a convex subalgebra which is wild almost concealed-canonical) and Meltzer~\cite{Meltzer} (every wild almost concealed-canonical algebra has a convex subcategory which is wild concealed). Thus in the paper we concentrate on the proof of the implication $(1) \implies (2)$.

The paper is organized as follows. In Section~\ref{section quivers} we recall basic facts about quivers and their representations. Section~\ref{section quasi tilted} is devoted to a short introduction of quasi-tilted algebras. Next in Sections~\ref{section varieties} and~\ref{section semiinvariants} we introduce module varieties and semi-invariants, respectively. Moreover, in Section~\ref{section semiinvariants} some reduction results for semi-invariants are obtained. Finally, in Section~\ref{section semiinvariants} we recall King's construction of moduli spaces and in Section~\ref{section final} we prove that the moduli spaces for the quasi-tilted algebras are products of projective spaces.

The paper was written during the author's stay at the University of Bielefeld, which was supported by CRC 701. The author also acknowledges the support of National Science Center Grant No.\ DEC-2011/03/B/ST1/00847.

\section{Quivers and their representations} \label{section quivers}

In this section we present facts about quivers and their representations, which we use in the paper. As a general background we suggest~\cites{AssemSimsonSkowronski, AuslanderReitenSmalo, Ringel}.

By a quiver $Q$ we mean a finite set $Q_0$ (called the set of vertices of $Q$) together with a finite set $Q_1$ (called the set of arrows of $Q$) and two maps $s, t : Q_1 \to Q_0$, which assign to each arrow $\alpha$ its starting vertex $s \alpha$ and its terminating vertex $t \alpha$, respectively. By a path of length $n \in \bbN_+$ in a quiver $Q$ we mean a sequence $\sigma = (\alpha_1, \ldots, \alpha_n)$ of arrows such that $s \alpha_i = t \alpha_{i + 1}$ for each $i \in [1, n - 1]$. In the above situation we put $\ell \sigma := n$, $s \sigma := s \alpha_n$ and $t \sigma := t \alpha_1$. We treat every arrow in $Q$ as a path of length $1$. Moreover, for each vertex $x$ we have a trivial path $\bone_x$ at $x$ such that $\ell \bone_x := 0$ and $s \bone_x := x =: t \bone_x$. For the rest of the paper we assume that the considered quivers do not have oriented cycles, where by an oriented cycle we mean a path $\sigma$ of positive length such that $s \sigma = t \sigma$.

Let $Q$ be a quiver. By the path algebra $\Bbbk Q$ of $Q$ we mean the vector space with a basis formed by the paths in $Q$ and the multiplication induced by the concatenation of paths. If $x$ and $y$ are vertices of $Q$, we put $\Bbbk Q (x, y) := \bone_y \Bbbk Q \bone_x$, i.e.\ $\Bbbk Q (x, y)$ is the space spanned by the paths with the starting vertex $x$ and the terminating vertex $y$. The $\Bbbk Q$-modules may be identified with the $\Bbbk$-representations of $Q$, where by a $\Bbbk$-representation of $Q$ we mean $V$ consisting of finite dimensional $\Bbbk$-vector spaces $V (x)$, $x \in Q_0$, and $\Bbbk$-linear maps $V (\alpha) : V (s \alpha) \to V (t \alpha)$. In particular, if $M$ is a $\Bbbk Q$-module and $V$ is the corresponding representation, then $V (x) := \bone_x M$ for each $x \in Q_0$. We will usually identify $\Bbbk Q$-modules with the corresponding representations of $\Bbbk$. If $V$ and $W$ are representations of a quiver $Q$, then a morphism $\varphi : V \to W$ is given by linear maps $\varphi (x) : V (x) \to W (x)$, $x \in Q_0$, such that $W (\alpha) \varphi (s \alpha) = \varphi (t \alpha) V (\alpha)$ for each $\alpha \in Q_1$. We denote the category of $\Bbbk$-representations of $Q$ by $\rep Q$. If $V$ is a representation, $x, y \in Q_0$, then one defines $V (\omega) : V (x) \to V (y)$ in an obvious way. Given a representation $V$ of $Q$ we denote by $\bdim V$ its dimension vector defined by the formula $(\bdim V) (x) := \dim_\Bbbk V (x)$, for $x \in Q_0$. Observe that $\bdim V \in \bbN^{Q_0}$ for each representation $V$ of $Q$. We call the elements of $\bbN^{Q_0}$ dimension vectors. If $\bd$ is a dimension vector, then we denote by $\supp \bd$ the subquiver of $Q$ induced by the vertices $x$ such that $\bd (x) \neq 0$. A dimension vector $\bd$ is called connected if the quiver $\supp \bd$ is connected. A dimension vector $\bd$ is called sincere if $\supp \bd = Q$. If $V$ is a representation of $Q$, then we also denote $\supp \bdim V$ by $\supp V$ and call it the support of $V$.

By a bound quiver $(Q, I)$ we mean a quiver $Q$ together with an ideal $I$ of $\Bbbk Q$ such that $I \subseteq (\Bbbk Q_1)^2$, where by $\Bbbk Q_1$ we denote the ideal of $\Bbbk Q$ generated by the arrows. Given a bound quiver $(Q, I)$ we call the algebra $\Bbbk Q / I$ the path algebra of $(Q, I)$. Note that if $(Q, I)$ is a bound quiver, then the $\Bbbk Q / I$-modules may be identified with the representations $V$ of $Q$ such that $V (\omega) = 0$ for each $\omega \in I \cap (\bigcup_{x, y \in Q_0} \Bbbk Q (x, y))$. If $\Lambda$ is the path algebra of a bound quiver $(Q, I)$, then we call $Q$ the Gabriel quiver of $\Lambda$. Gabriel proved that (up to isomorphism) $Q$ is uniquely determined by $\Lambda$. Moreover, Gabriel's Theorem implies that each quasi-tilted algebra is Morita equivalent to the path algebra of a bound quiver (we also need~\cite{HappelReitenSmalo}*{Proposition~III.1.1(b)} for this result). Thus from now on all considered algebras are the path algebras of bound quivers. Observe that if $J$ is an ideal in an algebra $\Lambda$, then the Gabriel quiver of $\Lambda / J$ is a subquiver of the Gabriel quiver of $\Lambda$ (here this is important that there are no oriented cycles in the considered quivers). If $(Q, I)$ is a bound quiver, then an algebra $\Lambda'$ is called a convex subalgebra of $\Bbbk Q / I$ if there exists a convex subquiver $Q'$ of $Q$ such that $\Lambda' = \Bbbk Q' / (I \cap \Bbbk Q')$.

Let $\Lambda$ be an algebra with the Gabriel quiver $Q$. For a vertex $x$ of $Q$ we put $P_\Lambda (x) := \Lambda \bone_x$. Then $P_\Lambda (x)$ is an indecomposable projective $\Lambda$-module and every indecomposable projective $\Lambda$-module is (up to isomorphism) of this form. If $V$ is a $\Lambda$-module, then $\Hom_\Lambda (P_\Lambda (x), V) = V (x)$. In particular, $\Hom_\Lambda (P_\Lambda (x), P_\Lambda (y)) = \Bbbk Q (y, x)$. Moreover, if $x, y \in Q_0$, $\omega \in \Bbbk (y, x)$ and $V$ is a $\Lambda$-module, then
\[
\Hom_\Lambda (\omega, V) : \Hom_\Lambda (P_\Lambda (y), V) \to \Hom_\Lambda (P_\Lambda (x), V)
\]
is just $V (\omega) : V (y) \to V (x)$.

For an algebra $\Lambda$ we denote by $\mod \Lambda$ the category of $\Lambda$-modules. Next, if $\Lambda$ is an algebra, then we denote by $\Du_\Lambda$ the duality $\mod \Lambda \to \mod \Lambda^{\op}$, where $\Lambda^{\op}$ is the opposite algebra of $\Lambda$, given by
\[
\Du_\Lambda (M) := \Hom_k (M, k) \qquad (M \in \mod \Lambda).
\]
Finally, for an algebra $\Lambda$ we denote by $\tau_\Lambda$ the corresponding Auslander--Reiten translation, which assigns to each $\Lambda$-module $M$ another $\Lambda$-module $\tau_\Lambda M$ (see~\cite{AssemSimsonSkowronski}*{Section~IV.2} for a definition). We will need the following consequence of the
Auslander--Reiten formula~\cite{AssemSimsonSkowronski}*{Theorem~IV.2.13}: if $M$ and $N$ are $\Lambda$-modules and $\pdim_\Lambda M \leq 1$, then
\begin{equation} \label{ARformula1}
\dim_\Bbbk \Ext_\Lambda^1 (M, N) = \dim_\Bbbk \Hom_\Lambda (N, \tau M).
\end{equation}

Let $\Lambda$ be an algebra with the Gabriel quiver $Q$. Since there are no cycles in $Q$, $\gldim \Lambda < \infty$. Consequently, we may define the bilinear form $\langle -, - \rangle_\Lambda : \bbZ^{Q_0} \times \bbZ^{Q_0} \to \bbZ$ by the condition:
\[
\langle \bdim M, \dim N \rangle_\Lambda = \sum_{i \in \bbN} \dim_k
\Ext_\Lambda^i (M, N)
\]
for all $\Lambda$-modules $M$ and $N$. We denote the corresponding quadratic form, called the Euler from, by $\chi_\Lambda$.

\section{Quasi-tilted algebras} \label{section quasi tilted}

A module $T$ over an algebra $\Lambda$ is called tilting if $\pdim_\Lambda T \leq 1$, $\Ext_\Lambda^1 (T, T) = 0$, and $T$ is a direct sum of $n$ pairwise non-isomorphic indecomposable $\Lambda$-modules, where $n$ is the number of vertices of the Gabriel quiver of $\Lambda$. By a tilted algebra we mean the opposite algebra of the endomorphism algebra of a tilting module over the the path algebra of a quiver.

An algebra $\Lambda$ is called quasi-tilted if $\Lambda$ is the opposite algebra of the endomorphism algebra of a tilting object in a connected hereditary abelian $\Bbbk$-category with finite dimensional homomorphism and extension spaces. Two prominent examples of quasi-tilted algebras are the following: the tilted algebras introduced above and the Ringel canonical algebras $\Lambda (\bm, \blambda)$, where $\bm = (m_1, \ldots, m_n)$, $n \geq 3$, is a sequence of integers greater than $1$ and $\blambda = (\lambda_3, \ldots, \lambda_n)$. In the above situation $\Lambda (\bm, \blambda)$ is the path algebra of the quiver
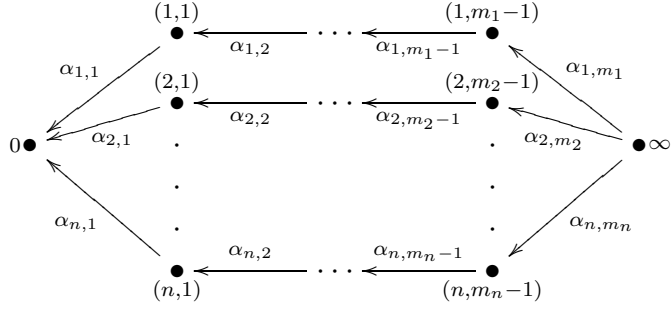
\[
\xymatrix@R=0.25\baselineskip@C=3\baselineskip{%
& \bullet \save*+!D{\scriptstyle (1, 1)} \restore
\ar[lddd]_{\alpha_{1, 1}} & \cdots \ar[l]^-{\alpha_{1, 2}} &
\bullet \save*+!D{\scriptstyle (1, m_1 - 1)} \restore
\ar[l]^-{\alpha_{1, m_1 - 1}}
\\ \\ %
& \bullet \save*+!D{\scriptstyle (2, 1)} \restore
\ar[ld]^{\alpha_{2, 1}} & \cdots \ar[l]^-{\alpha_{2, 2}} & \bullet
\save*+!D{\scriptstyle (2, m_2 - 1)} \restore \ar[l]^-{\alpha_{2,
m_2 - 1}}
\\ %
\bullet \save*+!R{\scriptstyle 0} \restore & \cdot & & \cdot &
\bullet \save*+!L{\scriptstyle \infty} \restore
\ar[luuu]_{\alpha_{1, m_1}} \ar[lu]^{\alpha_{2, m_2}}
\ar[lddd]^{\alpha_{n, m_n}}
\\ %
& \cdot & & \cdot
\\ %
& \cdot & & \cdot
\\ %
& \bullet \save*+!U{\scriptstyle (n, 1)} \restore
\ar[luuu]^{\alpha_{n, 1}} & \cdots \ar[l]_-{\alpha_{n, 2}} &
\bullet \save*+!U{\scriptstyle (n, m_n - 1)} \restore
\ar[l]_-{\alpha_{n, m_n - 1}} }
\]
modulo the ideal generated by the relations
\[
\alpha_{1, 1} \cdots \alpha_{1, m_1} + \lambda_i \alpha_{2, 1}
\cdots \alpha_{2, m_2} - \alpha_{i, 1} \cdots \alpha_{i, m_i}, \,
i \in [3, n].
\]
Due to Happel~\cite{Happel}*{Theorem~3.1} every quasi-tilted algebra is either a tilted algebra or is of canonical type (i.e.\ is derived equivalent to a canonical algebra).

A structure of module categories over tilted algebras has been investigated in~\cite{Kerner}, while a structure of module categories over quasi-tilted algebras of canonical type has been studied in~\cite{LenzingSkowronski}. We also refer to~\cite{Skowronski} for a characterization of tame quasi-tilted algebras and to~\cite{Ringel} for a description of module categories over so-called tubular algebras, which form an important subclass of tame quasi-tilted algebras. We list some consequences of these investigations.

Let $\Lambda$ be a tame quasi-tilted algebra with the Gabriel quiver $Q$. If $\bd$ is a dimension vector, then there exists an indecomposable $\Lambda$-module with dimension vector $\bd$ if and only if $\bd$ is a root of $\chi_\Lambda$, i.e.\ $\bd$ is a connected non-zero dimension vector such that $\chi_\Lambda (\bd) \in \{ 0, 1 \}$. We call a root $\bd$ isotropic if $\chi_\Lambda (\bd) = 0$. We call a root $\bd$ a Schur root if there exists a $\Lambda$-module $X$ (necessarily indecomposable) with dimension vector $\bd$ and trivial endomorphism algebra. If $X_1$ and $X_2$ are indecomposable $\Lambda$-modules such that $\bdim X_1$ and $\bdim X_2$ are isotropic roots with $\supp X_1 \cap \supp X_2 \neq \varnothing$, then either $\Hom_\Lambda (X_1, X_2) \neq 0$ or $\Hom_\Lambda (X_2, X_1) \neq \varnothing$ or $\bdim X_1$ and $\bdim X_2$ are multiplicities of the same isotropic Schur root.

Now assume that $\Lambda$ is a canonical algebra. The indecomposable $\Lambda$-modules can be divided into three classes $\calL_\Lambda$, $\calT_\Lambda$ and $\calR_\Lambda$: the class $\calL_\Lambda$ is formed by the indecomposable $\Lambda$-modules $X$ such that $\dim X (0) > \dim X (\infty)$, the class $\calT_\Lambda$ is formed by the indecomposable $\Lambda$-modules $X$ such that $\dim X (0) = \dim X (\infty)$, and the class $\calR_\lambda$ is formed by the indecomposable $\Lambda$-modules $X$ such that $\dim X (0) < \dim X (\infty)$. An algebra $\Lambda$ is called (almost) concealed-canonical if $\Lambda$ is the opposite algebra of the endomorphism algebra of a tilting module $T$, which is a direct sum of indecomposable modules from $\calL$ ($\calL \cup \calT$, respectively).

\section{Module varieties} \label{section varieties}

Let $\Lambda$ be the path algebra of a bound quiver $(Q, I)$. For a dimension vector $\bd$ we denote by $\mod_\Lambda (\bd)$ the set of representations $M$ of $(Q, I)$ (recall that we identify the $\Lambda$-modules with the representations of $(Q, I)$) such that $M (x) = \Bbbk^{\bd (x)}$ for each $x \in Q_0$. This set can be naturally identified with a closed subset of the affine space $\rep_Q (\bd) := \prod_{\alpha \in Q_1} \bbM (\bd (t \alpha), \bd (s \alpha))$, thus it has a structure of an affine variety (note that under this identification $\rep_Q (\bd) = \mod_{\Bbbk Q} (\bd)$). The reductive group $\GL (\bd) := \prod_{x \in Q_0} \GL (\bd (x))$ acts on $\mod_\Lambda (\bd)$ via:
\[
(g \ast M) (\alpha) := g (t \alpha) \cdot M (\alpha) \cdot g (s \alpha)^{-1} \qquad (g \in \GL (\bd), \; \alpha \in Q_1).
\]
If $M \in \mod_\Lambda (\bd)$, then we denote its orbit with respect to this action by $\calO (M)$. One has $\calO (M) = \calO (N)$ if and only if $M \simeq N$.

Let $\calC_1$ and $\calC_2$ be closed irreducible subsets of varieties $\mod_\Lambda (\bd_1)$ and $\mod_\Lambda (\bd_2)$, respectively. By $\calC_1 \oplus \calC_2$ we denote the closure of the set consisting of all $M \in \mod_\Lambda (\bd_1 + \bd_2)$ such that $M \simeq M_1 \oplus M_2$ for some $M_1 \in \calC_1$ and $M_2 \in \calC_2$. In the above situation we call $\calC_1$ and $\calC_2$ summands of $\calC$.

An irreducible component $\calC$ of $\mod_\Lambda (\bd)$ is called indecomposable if the indecomposable modules in $\calC$ form a dense subset of $\calC$. If $\calC$ is an irreducible component of $\mod_\Lambda (\bd)$, then there exist uniquely (up to ordering) determined indecomposable irreducible components $\calC_1$, \ldots, $\calC_n$ of $\calC$ such that
\[
\calC = \calC_1 \oplus \cdots \oplus \calC_n
\]
\cite{Crawley-BoeveySchroer}*{Theorem~1.1} (see also~\cite{delaPena}). We call the above presentation the generic decomposition of $\calC$. Moreover, if, for $i \in [1, n]$, $\calC_i \subseteq \mod_\Lambda (\bd_i)$, then we call $\bd_1$, \ldots, $\bd_n$ the generic summands of $\bd$ at $\calC$.

Now we present description of the indecomposable irreducible components in the case of the tame quasi-tilted algebras, which follows from~\cite{BobinskiSkowronski}. First, if $\bd$ is a dimension vector, then there is at most one indecomposable irreducible component of $\mod_\Lambda (\bd)$. Thus if it exists we denote it by $\calC (\bd)$. Moreover, there exists an indecomposable irreducible component of $\mod_\Lambda (\bd)$ if and only if $\bd$ is a Schur root. Moreover, if $\bd$ is not isotropic, then $\calC (\bd)$ is an orbit closure, i.e.\ there exists a $\Lambda$-module $M$ such that $\calC (\bd) = \ol{\calO (M)}$.

\section{Semi-invariants} \label{section semiinvariants}

Let $Q$ be a quiver, $\bd$ a dimension vector, and $\calC$ a $\GL (\bd)$-invariant closed subset of $\rep_Q (\bd)$. The action of $\GL (\bd)$ on $\calC$ induces an action on the coordinate ring $\Bbbk [\calC]$ via:
\[
(g \ast f) (M) := f (g^{-1} \ast M) \qquad (g \in \GL (\bd), \; f \in \Bbbk [\calC], \; M \in \calC).
\]
If $\calC$ is irreducible, then there is a unique closed orbit in $\calC$, that of the semi-simple module with dimension vector $\bd$, hence there are only trivial $\GL (\bd)$-regular functions on $\calC$, i.e.\ $\Bbbk [\calC]^{\GL (\bd)} = \Bbbk$. However, one may still have non-trivial semi-invariants. A regular function $f \in \Bbbk [\calC]$ is called a semi-invariant of weight $\theta \in \bbZ^{Q_0}$ if
\[
g \ast f = \chi^\theta (g) \cdot f
\]
for each $g \in \GL (\bd)$. Here $\chi^\theta : \GL (\bd) \to \Bbbk^\times$ is given by
\[
\chi^\theta (g) := \prod_{x \in Q_0} \det\nolimits^{\theta (\be_x)} (g (x)) \qquad (g \in \GL (\bd)),
\]
where $\be_x$, $x \in Q_0$, are the standard basis vectors of $\bbZ^{Q_0}$. We denote the space of semi-invariants of weight $\theta$ by $\Bbbk [\calC]_\theta$. One easily observes that $\theta (\bd) = 0$ provided $\Bbbk [\calC]_\theta \neq 0$.

We present a method of constructing semi-invariants, which in the case of quivers is due to Schofield~\cite{Schofield}, and has been generalized to the case of bound quivers independently by Derksen and Weyman~\cite{DerksenWeyman2002} and Domokos~\cite{Domokos} (we also refer to~\cites{DerksenWeyman2002, Domokos} for the proofs). Let $Q$ be a quiver and $\bd$ a dimension vector. Fix sequences $\bx = (x_1, \ldots, x_n)$ and $\by = (y_1, \ldots, y_m)$ of vertices of $Q$. Put
\[
\Bbbk Q (\bx, \by) := \prod_{\substack{i \in [1, n] \\ j \in [1, m]}} \Bbbk Q (x_i, y_j).
\]
If $\phi = (\phi_{i, j})_{\substack{i \in [1, n] \\ j \in [1, m]}} \in \Bbbk Q (\bx, \by)$ and $M \in \rep_Q (\bd)$ for a dimension vector $\bd$, then we obtain a map
\[
M (\phi) := [M (\phi_{i, j})]_{\substack{i \in [1, n] \\ j \in [1, m]}} : M (\bx) := \bigoplus_{i \in [1, n]} M (x_i) \to M (\by) := \bigoplus_{j \in [1, m]} M (y_j).
\]
If, in addition, $\sum_{i \in [1, n]} \bd (x_i) = \sum_{j \in [1, m]} \bd (y_j)$, then we may define a regular function $c_\bd^\phi : \rep_Q (\bd) \to \Bbbk$ by
\[
c_\bd^\phi (M) := \det M (\phi) \qquad (M \in \rep_Q (\bd)).
\]
Then $c_\bd^\phi$ is a semi-invariant of weight $\theta^\phi$, where
\[
\theta^\phi (\bc) := \sum_{i \in [1, n]} \bc (x_i) - \sum_{j \in [1, m]} \bc (y_j) \qquad (\bc \in \bbZ^{Q_0})
\]
(note that, in particular, $\theta^\phi (\bd) = 0$). If $\calC$ is a $\GL (\bd)$-invariant closed subset of $\rep_Q (\bd)$, then we denote the restriction $c_\bd^\phi |_\calC$ of $c_\bd^\phi$ to $\calC$ by $c_\calC^\phi$.

We list some obvious consequences (compare~\cite{DerksenWeyman2000}*{Lemma~1}).

\begin{lemma} \label{lemma direct sum}
Let $\bx$ and $\by$ be sequences of vertices of a quiver $Q$, $\phi \in \Bbbk Q (\bx, \by)$, $\bd$ a dimension vector such that $\theta^\phi (\bd) = 0$, and $M = M_1 \oplus M_2 \in \mod_\Lambda (\bd)$.
\begin{enumerate}

\item \label{point direct sum1}
If $\theta^\phi (\bdim M_1) \neq 0$ \textup{(}hence, equivalently, $\theta^\phi (\bdim M_2) \neq 0$\textup{)}, then $c_\bd^\phi (M) = 0$.

\item \label{point direct sum2}
If $\theta^\phi (\bdim M_1) = 0$ \textup{(}hence, equivalently, $\theta^\phi (\bdim M_2) = 0$\textup{)}, then $c_\bd^\phi (M) = c_{\bdim M_1}^\phi (M_1) \cdot c_{\bdim M_2}^\phi (M_2)$. \qed

\end{enumerate}
\end{lemma}

If we have sequences $\bx$, $\bx'$, $\by$ and $\by'$ of vertices of a quiver $Q$, $\phi \in \Bbbk (\bx, \by)$ and $\phi' \in \Bbbk (\bx', \by')$, then we may define the element $\phi \oplus \phi' \in \Bbbk Q (\bx \cdot \bx', \by \cdot \by')$ in the obvious way, where $\bx \cdot \bx'$ and $\by \cdot \by'$ are the concatenations of the respective sequences. Observe that
\[
M (\phi \oplus \phi') =
\begin{bmatrix}
M (\phi) & 0 \\ 0 & M (\phi')
\end{bmatrix}
: M (\bx) \oplus M (\bx') \to M (\by) \oplus M (\by')
\]
for each $M \in \rep_Q (\bd)$. Consequently, we get the following.

\begin{lemma} \label{lemma direct sum bis}
Let $\bx$, $\bx'$, $\by$ and $\by'$ be sequences of vertices of a quiver $Q$, $\phi \in \Bbbk (\bx, \by)$ and $\phi' \in \Bbbk (\bx', \by')$. If $\bd$ is a dimension vector and $\theta^\phi (\bd) = 0 = \theta^{\phi'} (\bd)$, then
\[
c_\bd^{\phi \oplus \phi'} = c_\bd^\phi \cdot c_\bd^{\phi'}.
\]
In particular, $c_\bd^{\phi \oplus \phi'}$ is a semi-invariant of weight $\theta^\phi + \theta^{\phi'}$.  \qed
\end{lemma}

We can interpret the above construction using projective presentations. Let $\Lambda$ be a factor algebra of $\Bbbk Q$ for a quiver $Q$. As above let $\bx= (x_1, \ldots, x_n)$ and $\by = (y_1, \ldots, y_m)$ be sequences of vertices of a quiver $Q$, and $\phi \in \Bbbk Q (\bx, \by)$. If we put
\[
P_\Lambda (\bx) := \bigoplus_{i \in [1, n]} P_\Lambda (x_i) \qquad \text{and} \qquad P_\Lambda (\by) := \bigoplus_{j \in [1, m]} P_\Lambda (y_j),
\]
then we may view $\phi$ as a map $P_\Lambda (\by) \to P_\Lambda (\bx)$ (note that every map between projective $\Lambda$-modules is of this form some $\bx$, $\by$ and $\phi$). Observe that
\[
\theta^\phi (\bdim M) = \dim_{\Bbbk} \Hom_\Lambda (P_\Lambda (\bx), M) - \dim_{\Bbbk} \Hom_\Lambda (P_\Lambda (\by), M)
\]
for each $\Lambda$-module $M$. If $M \in \mod_\Lambda (\bd)$ for a dimension vector $\bd$, then $M (\phi)$ may be identified with the induced map
\[
\Hom_\Lambda (\phi, M) : \Hom_\Lambda (P_\Lambda (\bx), M) \to
\Hom_\Lambda (P_\Lambda (\by), M).
\]
This implies in particular, that if $\theta^\phi (\bd) = 0$, then $c_\bd^\phi (M) \neq 0$ if and only if $\Hom_\Lambda (\Coker \phi, M) = 0$. In order to associate a semi-invariant with $\Coker \phi$ (independently of $\phi$) we make some additional assumptions, which simplify the presentation.

Let $\Lambda$ be an algebra with the Gabriel quiver $Q$. Moreover, let $V$ be a $\Lambda$-module with projective dimension at most $1$. If $\phi : Q \to P$ is a projective presentation of $V$ such that $\phi$ is a monomorphism, then $\theta^\phi = \langle \bdim V, - \rangle_\Lambda$, hence is independent on $\phi$. We denote this weight by $\theta^V$. Moreover, if $\theta^V (\bd) = 0$, and $\phi$ and $\phi'$ are projective presentations of $V$ such that $\phi$ and $\phi'$ are monomorphisms, then $c_{\mod_\Lambda (\bd)}^\phi$ and $c_{\mod_\Lambda (\bd)}^{\phi'}$ coincide up to a non-zero scalar. Thus we may define $c_{\Lambda, \bd}^V \in \Bbbk [\mod_\Lambda (\bd)]$ by $c_{\Lambda, \bd}^V := c_{\mod_\Lambda (\bd)}^\phi$, where $\phi$ is a chosen projective presentation of $V$. Then $c_{\Lambda, \bd}^V$ is a semi-invariant of weight $\theta^V$ and $c_{\Lambda, \bd}^V (M) = 0$ if and only if $\Hom_\Lambda (V, M) = 0$ or, equivalently, $\Hom_\Lambda (M, \tau_\Lambda V) = 0$ (for the latter statement we need~\eqref{ARformula1}). If $\calC$ is a closed $\GL (\bd)$-invariant subset of $\mod_\Lambda (\bd)$, then we denote by $c_{\Lambda, \calC}^V$ the restriction of $c_{\Lambda, \bd}^V$ to $\calC$.

It will be often useful to associate semi-invariants to modules of projective dimension at most 1 in a ``regular'' way. Thus assume in addition that $\bc$ is a dimension vector such that $\calP_\Lambda (\bc) \neq \varnothing$, where $\calP_\Lambda (\bc)$ is the subset of $\mod_\Lambda (\bc)$ consisting of the modules with projective dimension at most $1$. Let $\bx$ be a sequence of vertices of $Q$ such that $P_\Lambda (\bx) = \bigoplus_{x \in Q_0} P_\Lambda (x)^{\bc (x)}$. Since there exist an epimorphism $P_\Lambda (\bx) \twoheadrightarrow V$ for each $\Lambda$-module $V$ with dimension vector $\bc$ and $\calP_\Lambda (\bc) \neq \varnothing$, there exists a sequence $\by$ of vertices of $Q$ such that $\bdim P_\Lambda (\bx) - \bdim P_\Lambda (\by) = \bc$. Let $\calX_\Lambda (\bc) \subseteq \Bbbk Q (\bx, \by)$ be the set of monomomorphisms $P_\Lambda (\by) \hookrightarrow P_\Lambda (\bx)$. Obviously, if $\phi \in \calX_\Lambda (\bc)$, then $\Coker \phi \in \calP_\Lambda (\bc)$. On the other hand, if $V \in \calP_\Lambda (\bc)$, then there exists $\phi \in \calX_\Lambda (\bd)$ such that $\Coker \phi \simeq V$. In fact we have even more.

\begin{lemma} \label{lemma open}
Let $\Lambda$ be an algebra and $\bc$ a dimension vector. If $\calV$ is a non-empty open subset of $\calP_\Lambda (\bc)$ and
\[
\calU := \{ \phi \in \calX_\Lambda (\bc) : \text{$\Coker \phi \in \calV$} \},
\]
then $\calU$ is a non-empty open subset of $\calX_\Lambda (\bc)$.
\end{lemma}

\begin{proof}
Let $\calY_\Lambda (\bc)$ be the set of $\psi = (\psi (x))_{x \in Q_0}$ such that, for each $x \in Q_0$, $\psi (x) : P_\Lambda (\bc) (x) \to \Bbbk^{\bc (x)}$ is a linear map. We denote by $\calZ_\Lambda (\bc)$ the set of pairs $(\phi, \psi)$ such that $\phi \in \calX_\Lambda (\bc)$, $\psi \in \calY_\Lambda (\bc)$ and $\psi \circ \phi = 0$. If $\pi : \calZ_\Lambda (\bc) \to \calX_\Lambda (\bc)$ is the canonical projection, then $\pi$ is a vector bundle. Consequently, if $\calZ_\Lambda' (\bc)$ is the set of pairs $(\phi, \psi) \in \calZ_\Lambda (\bc)$ such that $\psi$ is a surjection and $\pi'$ is the restriction of $\pi$ to
$\calZ_\Lambda' (\bc)$, then $\pi'$ is locally trivial (with fibre isomorphic to $\GL (\bc)$). In particular, if $\calW$ is an open subset of $\calZ_\Lambda' (\bc)$, then $\pi' (\calW)$ is an open subset of $\calX_\Lambda (\bc)$.

There exists a regular map $\Theta : \calZ_\Lambda' (\bc) \to \mod_\Lambda (\bc)$ such that $\Theta (\phi, \psi) \simeq \Coker \phi$ for all $(\phi, \psi) \in \calZ_\Lambda' (\bc)$ (the proof is analogous to the proof of~\cite{Richmond}*{Lemma~9}, hence we omit it). Since $\calU = \pi (\Theta^{-1} (\calV))$, the claim follows.
\end{proof}

We will also need the following.

\begin{lemma} \label{lemma generate}
Let $\Lambda$ be an algebra, $\bd$ a dimension vector, $\calC$ a $\GL (\bd)$-invariant irreducible closed subset of $\mod_\Lambda (\bd)$, and $\bc$ a dimension vector such that $\calP_\Lambda (\bc) \neq \varnothing$.
\begin{enumerate}

\item \label{point generate1}
If $\calU$ is a non-empty open subset of $\calX_\Lambda (\bc)$, then
\[
\spn \{ c_\calC^\phi \mid \text{$\phi \in \calU$} \} = \spn \{ c_\calC^\phi \mid \text{$\phi \in \calX_\Lambda (\bc)$} \}.
\]

\item \label{point generate2}
If $\calV$ is a non-empty open subset of $\calP_\Lambda (\bc)$, then
\[
\spn \{ c_{\Lambda, \calC}^V \mid \text{$V \in \calV$} \} = \spn \{ c_{\Lambda, \calC}^V \mid \text{$V \in \calP_\Lambda (\bc)$}
\}.
\]

\end{enumerate}
\end{lemma}

\begin{proof}
Using Lemma~\ref{lemma open}, it is sufficient to prove the first assertion. Let $\phi_1, \ldots, \phi_m \in \calX_\Lambda (\bc)$ be such that $c_{\calC}^{\phi_1}$, \ldots, $c_\calC^{\phi_m}$ form a basis of $\spn \{ c_\calC^\phi \mid \text{$\phi \in \calX_\Lambda (\bc)$} \}$. There exist $M_1, \ldots, M_m \in \calC$ such that
\[
\det [c_{\calC}^{\phi_i} (M_j)]_{1 \leq i, j \leq m} \neq 0.
\]
It suffices to show there exist $\psi_1, \ldots, \psi_m \in \calU$ such that
\[
\det [c_{\calC}^{\psi_i} (M_j)]_{1 \leq i, j \leq m} \neq 0.
\]
However, the regular function
\[
\phi : \calX_\Lambda (\bc)^m \to \Bbbk, \; (\psi_1, \ldots, \psi_m) \mapsto \det [c_{\calC}^{\psi_i} (M_j)]_{1 \leq i, j \leq
m},
\]
is not a zero function, hence the claim follows.
\end{proof}

Now we use the above construction to describe generating sets of semi-invariants. We present two such sets. Depending on the situation, this will be more convenient to use one of them.

Let $\Lambda$ be an algebra with the Gabriel quiver $Q$, $\bd$ a dimension vector, $\calC$ an irreducible component of $\mod_\Lambda (\bd)$, and $\theta$ a weight such that $\SI [\calC]_\theta \neq 0$. There exists unique $\bc_\theta \in \bbZ^{Q_0}$ such that $\theta = \langle \bc_\theta, - \rangle_{\Bbbk Q}$. Since $\SI [\rep_Q (\bd)]_\theta \neq 0$, we may assume that $\bc_\theta$ is a dimension vector. We explain this more precisely.

If $\theta'$ and $\theta''$ are weights, then $\SI [\rep_Q (\bd)]_{\theta'}$ and $\SI [\rep_Q (\bd)]_{\theta''}$ are equal and both non-zero if and only if $\theta'$ and $\theta''$ are $\bd$-equivalent, i.e.\ $\theta' (x) = \theta'' (x)$ for all $x \in (\supp \bd)_0$. Now \cite{DerksenWeyman2000}*{Theorem~1} (see also~\cite{SchofieldvandenBergh}*{Theorem~2.3}) implies that there exists a dimension vector $\bc$ such that the weights $\theta$ and $\langle \bc, - \rangle_{\Bbbk Q}$ are $\bd$-equivalent. Consequently, we may assume that we only consider weights of this form.

It clear that $\calP_{\Bbbk Q} (\bc_\theta) \neq \varnothing$ (the category $\rep Q$ is hereditary), hence also $\calX_{\Bbbk Q} (\bc_\theta) \neq \varnothing$. It follows from~\cite{Chindris2009}*{Corollary~2.5} that the semi-invariants $c_\bd^\phi$, $\phi \in \calX_{\Bbbk Q} (\bc_\theta)$, span $\SI [\rep_Q (\bd)]_\theta$. Since $\calC$ is a closed $\GL (\bd)$-invariant subset of $\rep_Q (\bd)$ and $\chr \Bbbk = 0$, it follows that the semi-invariants $c_\calC^\phi$, $\phi \in \calX_{\Bbbk Q} (\bc_\theta)$, span $\Bbbk [\calC]_\theta$.

We list some consequences.

\begin{lemma} \label{lemma summand}
Let $\Lambda$ be an algebra, $\bd$ a dimension vector, $\calC$ an irreducible component of $\mod_\Lambda (\bd)$, and $\theta$ a weight such that $\SI [\calC]_\theta \neq 0$. If $\calC'$ is a summand of $\calC$, then $\SI [\calC']_{\theta} \neq 0$.
\end{lemma}

\begin{proof}
Write $\calC = \calC' \oplus \calC''$. By assumption there exists $\phi \in \calX_{\Bbbk Q} (\bc_\theta)$ and $M \in \calC$ such that $c_\calC^\phi (M) \neq 0$. Without loss of generality we may assume that $M = M' \oplus M''$ for $M' \in \calC'$ and $M'' \in \calC''$. Lemma~\ref{lemma direct sum}\eqref{point direct sum1} implies that $\theta (\bdim M_1) = 0 = \theta (\bdim M_2)$. Consequently, Lemma~\ref{lemma direct sum}\eqref{point direct sum2} implies that $c_\calC^\phi (M) = c_{\calC'}^\phi (M') \cdot c_{\calC''}^\phi (M'')$. In particular, $c_{\calC'}^\phi (M') \neq 0$.
\end{proof}

\begin{lemma} \label{lemma nonzero}
Let $\Lambda$ be an algebra, $\bd$ a dimension vector, $\calC$ an irreducible component of $\mod_\Lambda (\bd)$, and $\theta$ a weight such that $\SI [\calC]_\theta \neq 0$. Then there exists an open subset $\calU$ of $\calX_{\Bbbk Q} (\bc_\theta)$ such that $c_\calC^\phi \neq 0$ for all $\phi \in \calU$.
\end{lemma}

\begin{proof}
There exist $\phi_0 \in \calP_{\Bbbk Q} (\bc_\theta)$ and $M \in \calC$ such that $c_\calC^{\phi_0} (M) \neq 0$. We define a function $\Phi : \calX_{\Bbbk Q} (\bc_\theta) \to \Bbbk$ by
\[
\Phi (\phi) := c_{\calC}^\phi (M) \qquad (\phi \in \calX_{\Bbbk Q} (\bc_\theta)).
\]
This is a regular function and we take $\calU := \Phi^{-1} (\Bbbk^\times)$.
\end{proof}

Now we present the second construction. As above let $\Lambda$ be an algebra with the Gabriel quiver $Q$, $\bd$ a dimension vector, $\calC$ an irreducible component of $\mod_\Lambda (\bd)$, and $\theta$ a weight such that $\SI [\calC]_\theta \neq 0$. Next let $\Lambda_{\calC} := \Lambda / \Ann \calC$, where $\Ann \calC := \bigcap_{M \in \calC} \Ann M $. Consequently, $\calC$ is a faithful irreducible component of $\mod_{\Lambda_{\calC}} (\bd)$. Again there exists $\bc_{\theta, \calC} \in \bbZ^{Q_0}$ such that $\theta = \langle \bc_{\theta, \calC}, - \rangle_{\Lambda_{\calC}}$. Since $\Bbbk [\calC]_\theta \neq 0$, $\bc_{\theta, \calC}$ is a dimension vector and $\calP_{\calC} (\theta) := \calP_{\Lambda_{\calC}} (\bc_{\theta, \calC}) \neq \varnothing$ (see~\cite{DerksenWeyman2002}*{Theorem~1}). Moreover, \cite{DerksenWeyman2002}*{Theorem~1}~also says that $\Bbbk [\calC]_\theta$ is spanned by the semi-invariants $c_{\Lambda_\calC, \calC}^V$, $V \in \calP_\calC (\theta)$.

It is known, that $\ol{\calP}_\calC (\theta)$ is an irreducible component of $\mod_{\Lambda_{\calC}} (\bc_{\theta, \calC})$~\cite{BarotSchroer}*{Proposition~3.1}. It is quite easy to observe that the generic decomposition of $\ol{\calP}_\calC (\theta)$ is of the form
\[
\ol{\calP}_\calC (\theta) = \ol{\calP}_{\Lambda_{\calC}} (\bc_1) \oplus \cdots \oplus \ol{\calP}_{\Lambda_{\calC}} (\bc_n)
\]
for some dimension vectors $\bc_1$, \ldots, $\bc_n$, such that $\bc_{\theta, \calC} = \bc_1 + \cdots + \bc_n$. Obviously $\bc_1$, \ldots, $\bc_n$ are the generic summands of $\bc_{\theta, \calC}$ (at $\ol{\calP}_\calC (\theta)$). If we put $\theta_i := \langle \bc_i, - \rangle_{\Lambda_\calC}$, then we call the presentation
\[
\theta = \theta_1 + \ldots + \theta_n
\]
the generic decomposition of $\theta$ at $\calC$. Moreover, we call $\theta_1$, \ldots, $\theta_n$ the generic summands of $\theta$ at $\calC$.

As a first consequence we get the following.

\begin{lemma} \label{lemma decomposition}
Let $\Lambda$ be a triangular algebra, $\bd$ a dimension vector, $\calC$ an irreducible component of $\mod_\Lambda (\bd)$, and $\theta$ a weight such that $\SI [\calC]_\theta \neq 0$. If $\theta = \theta_1 + \cdots + \theta_n$ is the generic decomposition of $\theta$ at $\calC$, then the image of map
\[
\SI [\calC]_{\theta_1} \times \cdots \times \SI [\calC]_{\theta_n} \to \SI [\calC]_\theta, (f_1, \ldots, f_n) \mapsto f_1 \cdots f_n,
\]
spans $\SI [\calC]_\theta$. In particular, $\SI [\calC]_{\theta_i} \neq 0$ for each $i \in [1, n]$.
\end{lemma}

\begin{proof}
Let $\bc_1$, \ldots, $\bc_n$ be the dimension vectors corresponding to the weights $\theta_1$, \ldots, $\theta_n$, respectively, in the sense explained above.
The set $\calV$ of $V \in \calP_\calC (\theta)$ such that $V \simeq V_1 \oplus \cdots \oplus V_n$ for $V_1 \in \calP_{\Lambda_\calC} (\bc_1)$, \ldots, $V_n \in \calP_{\Lambda_\calC} (\bc_n)$, contains an open subset of $\calP_\calC (\theta)$. Lemma~\ref{lemma generate}\eqref{point generate2} implies that $\SI [\calC]_\theta$ is spanned by the semi-invariants $c_{\Lambda_{\calC}, \calC}^{V_1 \oplus \cdots \oplus V_n}$, $V_1 \in \calP_{\Lambda_\calC} (\bc_1)$, \ldots, $V_n \in \calP_{\Lambda_\calC} (\bc_n)$. Moreover, Lemma~\ref{lemma direct sum bis} implies that
\[
c_{\Lambda_{\calC}, \calC}^{V_1 \oplus \cdots \oplus V_n} = c_{\Lambda_{\calC}, \calC}^{V_1} \cdots c_{\Lambda_{\calC}, \calC}^{V_n}
\]
for all $V_1 \in \calP_{\Lambda_\calC} (\bc_1)$, \ldots, $V_n \in \calP_{\Lambda_\calC} (\bc_n)$, hence the claim follows.
\end{proof}

As a next consequence we obtain the following useful fact.

\begin{lemma} \label{lemma dimension}
Let $\Lambda$ be a triangular algebra, $\bd$ a dimension vector, $\calC$ an irreducible component of $\mod_\Lambda (\bd)$, and $\theta$ a weight such that $\SI [\calC]_\theta \neq 0$. If $\calP_{\calC} (\theta)$ contains a dense orbit, then $\dim_\Bbbk \SI [\calC]_\theta = 1$.
\end{lemma}

\begin{proof}
If $\calO (V)$ is a dense orbit in $\calP_{\calC} (\theta)$, then Lemma~\ref{lemma generate}\eqref{point generate2} implies that $\SI [\calC]_\theta$ is spanned by the semi-invariant $c_{\Lambda_\calC, \calC}^V$, hence the claim follows.
\end{proof}

Consequently, we get the following.

\begin{corollary} \label{corollary dense}
Let $\Lambda$ be a triangular algebra, $\bd$ a dimension vector, $\calC$ an irreducible component of $\mod_\Lambda (\bd)$, and $\theta$ a weight such that $\SI [\calC]_\theta \neq 0$. If $\bc'$ is a generic summand of $\bc_{\theta, \calC}$ at $\ol{\calP}_\calC (\theta)$ such that $\calP_{\Lambda_\calC} (\bc')$ contains a dense orbit, then
\[
\SI [\calC]_\theta \simeq \SI [\calC]_{\theta - \theta'},
\]
where $\theta' := \langle \bc', - \rangle_{\Lambda_\calC}$.
\end{corollary}

\begin{proof}
Lemma~\ref{lemma decomposition} implies that $\SI [\calC]_{\theta'} \neq 0$. Together with Lemma~\ref{lemma dimension} this implies that $\dim_\Bbbk \SI [\calC]_{\theta'} = 1$. Fix a non-zero semi-invariant $f \in \SI [\calC]_{\theta'}$. Lemma~\ref{lemma decomposition} implies that the map
\[
\SI [\calC]_{\theta - \theta'} \to \SI [\calC]_\theta, \; c \mapsto f \cdot c,
\]
is surjective. Since $\calC$ is irreducible, this is also injective, and the claim follows.
\end{proof}

\section{Moduli spaces} \label{section moduli}

Let $\Lambda$ be an algebra, $\bd$ a dimension vector, and $\calC$ an irreducible component of $\mod_\Lambda (\bd)$. If $\theta$ is a weight, then a $\Lambda$-module $M \in \calC$ is called $\theta$-semistable if there exists $f \in \SI [\calC]_{p \theta}$, for some $p \in \bbN_+$, such that $f (M) \neq 0$. King~\cite{King} has proved that $M$ is $\theta$-semistable if and only if $\theta (\bdim M) = 0$ and $\theta (\bdim N) \leq 0$ for each submodule $N$ of $M$. We denote by $\calC_\theta^{\sstab}$ the set of $\theta$-semistable $\Lambda$-modules in $\calC$. King has also constructed a coarse moduli $\calM (\calC)_\theta^{\sstab}$ for the $\theta$-semistable modules in $\calC$ (up to an equivalence, which identifies modules which have the same simple composition factors within the category of $\theta$-semistable modules). By definition
\[
\calM (\calC)_\theta^{\sstab} = \Proj \Bigl( \bigoplus_{p \in \bbN} \SI [\calC]_{p \theta} \Bigr).
\]

\begin{lemma} \label{lemma summand orbit}
Let $\Lambda$ be a triangular algebra, $\bd$ a dimension vector, $\calC$ an irreducible component of $\mod_\Lambda (\bd)$, and $\theta$ a weight such that $\calC_\theta^{\sstab} \neq \varnothing$. If $\calC = \calC_1 \oplus \calC_2$ for irreducible components $\calC_1$ and $\calC_2$ of $\mod_\Lambda (\bd_1)$ and $\mod_\Lambda (\bd_2)$, respectively, and $\calC_2$ is an orbit closure, then
\[
\calM (\calC)_\theta^{\sstab} \simeq \calM (\calC_1)_\theta^{\sstab}.
\]
\end{lemma}

\begin{proof}
Without loss of generality we may assume that $\SI [\calC]_\theta \neq 0$. Then we will show that
\[
\SI [\calC]_{p \theta} \simeq \SI [\calC_1]_{p \theta}
\]
for each $p \in \bbN$.

Let $\calC_2 = \ol{\calO (M)}$. Consider the map $\Phi : \calC_1 \to \calC$ given by
\[
\Phi (N) := N \oplus M \qquad (N \in \calC_1).
\]
We will show that $\Phi^* : \Bbbk [\calC] \to \Bbbk [\calC_1]$ induces an isomorphism $\Phi_p^* : \SI [\calC]_{p \theta} \simeq \SI [\calC_1]_{p \theta}$ for each $p \in \bbN$.

Fix $p \in \bbN$. Since $\GL (\bd) \times (\calC_1 \oplus \{ M \})$ is a dense subset of $\calC$, it is clear that $\Phi_p^*$ is a monomorphism. Thus it remains to show that $\Phi_p^*$ is an epimorphism. Let $\calX := \calX_{\Bbbk Q} (p \theta)$, where $Q$ is the Gabriel quiver of $\Lambda$. Using Lemma~\ref{lemma
generate}\eqref{point generate1}, it suffices to show that there exists an open subset $\calU$ of $\calX$ such that $c_{\calC_1}^\phi$ is in the image of $\Phi_p^*$ for each $\phi \in \calU$.

It follows from Lemma~\ref{lemma summand} that $\SI [\calC_2]_{p \theta} \neq 0$. Using Lemma~\ref{lemma nonzero}, we obtain, that there exists an open subset $\calU$ of $\calX$ such that $c_{\calC_2}^\phi \neq 0$ for each $\phi \in \calU$. In particular, $c_{\calC_2}^\phi (M) \neq 0$ for each $\phi \in \calU$. Now it follows from Lemma~\ref{lemma direct sum}, that
\[
c_{\calC_1}^\phi = \Phi_p^* \left( \frac{1}{c_{\calC_2}^\phi (M)} c_{\calC}^\phi \right)
\]
for each $\phi \in \calU$.
\end{proof}

\section{Moduli spaces for tame quasi-tilted algebras} \label{section final}

Let $\Lambda$ be a tame quasi-tilted algebra, $\bd$ a dimension vector, $\calC$ an irreducible component of $\mod_\Lambda (\bd)$, and $\theta$ a weight such that $\calC_\theta^{\sstab} \neq \varnothing$. The aim of this section is to prove that $\calM (\calC)_\theta^{\sstab}$ is a product of projective spaces. Let $\Lambda' := \Lambda / \Ann \calC$.

We know that there exist Schur roots $\bd_1$, \ldots, $\bd_n$ such that
\[
\calC = \calC_1 \oplus \cdots \oplus \calC_n,
\]
where $\calC_i := \calC (\bd_i)$, $i \in [1, n]$. Using Lemma~\ref{lemma summand orbit}, we may assume that $\bd_1$, \ldots, $\bd_n$ are isotropic.

Now let $\bc_1$, \ldots, $\bc_m$ be the generic summands of $\bc_{\theta, \calC}$ at $\ol{\calP}_{\Lambda'} (\bc_{\theta, \calC})$. Using Corollary~\ref{corollary dense} we may assume that, for each $j \in [1, m]$, $\calP_{\Lambda'} (\bc_j)$ does not contain a dense orbit. Since $\ol{\calP}_{\Lambda'} (\bc_j)$ is an indecomposable irreducible component, this implies that there exists infinitely many indecomposable $\Lambda'$-modules of dimension vector $\bc_j$, for each $j \in [1, m]$. This also means that, for each $j \in [1, m]$, there exists infinitely many indecomposable $\Lambda$-modules of dimension vector $\bc_j$, hence $\bc_1$, \ldots, $\bc_m$ are isotropic roots of $\chi_\Lambda$ (using more detailed knowledge of $\mod \Lambda$ one could also show that they are Schur roots, but we will not use this).

Before we formulate the next lemma let us recall that if $i \in [1, n]$, $j \in [1, m]$, $M \in \calC_i$ and $V \in \calP_{\Lambda'} (\bc_j)$, then $c_{\calC_i}^V (M) = 0$ if and only if $\Hom_\Lambda (V, M) \neq 0$ or $\Hom_\Lambda (M, \tau_{\Lambda'} V) \neq 0$.

\begin{lemma} \label{lemma support}
With the above notation either $\bc_j$ is a multiplicity of $\bd_i$ or $\supp \bd_i \cap \supp \bc_j = \varnothing$ for all $i \in [1, n]$ and $j \in [1, m]$.
\end{lemma}

\begin{proof}
Fix $i \in [1, n]$ and $j \in [1, m]$. Note that Lemmas~\ref{lemma summand} and~\ref{lemma decomposition} imply that $\SI [\calC_i]_{\theta'} \neq 0$, where $\theta' := \langle \bc_j, - \rangle_{\Lambda'}$. Assume that neither $\bc_j$ is a multiplicity of $\bd_i$ nor $\supp \bd_i \cap \supp \bc_j = \varnothing$. Then Section~\ref{section quasi tilted} implies that one of the following situations holds:
\begin{enumerate}

\item
$\Hom_\Lambda (V, M) \neq 0$, for each indecomposable $\Lambda$-module $M$ with dimension vector $\bd_i$ and each indecomposable $\Lambda$-module $V$ with dimension vector $\bc_j$,
or

\item
$\Hom_\Lambda (M, V) \neq 0$, for each indecomposable $\Lambda$-module $M$ with dimension vector $\bd_i$ and each indecomposable $\Lambda$-module $V$ with dimension vector $\bc_j$.

\end{enumerate}
In the first case we immediately obtain that $\SI [\calC (\bd_i)]_{\theta'} = 0$, contradiction. We show that we get the same conclusion in the second case.

Since $\Lambda$, hence also $\Lambda'$, is tame and there are infinitely many indecomposable $\Lambda'$-modules in $\calP_{\Lambda'} (\bc_j)$, \cite{CrawleyBoevey}*{Theorem~D} implies that there is a non-empty open subset $\calV$ of $\calP_{\Lambda'} (\bc_j)$ such that $\tau_{\Lambda'} V \simeq V$ for each $V \in \calV$. In particular, $\dim \tau_{\Lambda'} V = \bc_j$ for each $V \in \calV$. Thus (2) together with Lemma~\ref{lemma generate}\eqref{point generate2}implies that $\SI [\calC (\bd_i)]_{\theta'} = 0$, and this finishes the proof.
\end{proof}

Let $I'$ be the set of $i \in [1, n]$ such that there exists $j \in [1, m]$ with $\bc_j$ being a multiplicity of $\bd_i$. Let $I'' := [1, n] \setminus I$. Lemma~\ref{lemma support} implies that $\supp \bd_p \cap \supp \bd_q = \varnothing$ if $p \in I'$ and $q \in I''$ (since $\supp \bd_p = \supp \bc_j$ for some $j \in J$). Thus $\calC = \calC_1 \times \calC_2$, where
\[
\calC_1 := \bigoplus_{i \in I'} \calC (\bd_i) \qquad \text{and} \qquad \calC_2 := \bigoplus_{i \in I''} \calC (\bd_i).
\]
Consequently,
\[
\calM (\calC)_\theta^{\sstab} = \calM (\calC_1)_\theta^{\sstab} \times \calM (\calC_2)_\theta^{\sstab}.
\]
If $I' \neq [1, n] \neq I''$, then we get our claim by induction. Hence we have only two cases to consider: either $I' = [1, n]$ or $I'' = [1, n]$.

First assume that $I'' = [1, n]$. Since $\calC$ is a faithful component over $\Lambda'$, hence $\bd$ is a sincere dimension vector over $\Lambda'$. On the other hand, $\supp \bc_j \cap \supp \bd = \varnothing$, hence $\bc_j = 0$. Consequently, $\theta = 0$.
Thus
\[
\SI [\calC]_{p \theta} = \Bbbk [\calC]^{\GL (\bd)} = \Bbbk
\]
for each $p \in \bbN$, and
\[
\calM (\calC)_\theta^{\sstab} = \Proj (\Bbbk [T]) = \{ * \}.
\]

Now assume that $I' = [1, n]$. We can make another reduction in this case. Let $\bd_1'$, \ldots, $\bd_l'$ be the pairwise different vectors among $\bd_1$, \dots, $\bd_n$. For each $p \in [1, l]$, let $I_p$ be the set of $i \in [1, n]$ such that $\bd_i = \bd_p'$. Let $\calC_p' := \bigoplus_{i \in I_p} \calC (\bd_i)$ for $p \in [1, l]$. Lemma~\ref{lemma support} again implies that $\supp \bd_p' \cap \supp \bd_q' = \varnothing$ if $p, q \in [1, l]$ and $p \neq q$. Consequently,
\[
\calC = \calC_1' \times \cdots \times \calC_l'
\]
and
\[
\calM (\calC)_\theta^{\sstab} = \calM (\calC_1')_\theta^{\sstab} \times \cdots \times \calM (\calC_l')_\theta^{\sstab}.
\]
If $l > 1$, then the claim follows by induction again, thus we may assume $l = 1$. In this case~\cite{BobinskiSkowronski}*{Theorem~2} implies that $\mod_\Lambda (\bd)$ is irreducible, hence $\calC = \mod_\Lambda (\bd)$. Hence the claim follows from the following.

\begin{proposition}
Let $\Lambda$ be a tame quasi-tilted algebra and let $\bh$ be an isotropic Schur root of $\chi_\Lambda$. If $n, p \in \bbN_+$, then
\[
\calM (\mod_\Lambda (n \bh))_{p \langle \bh, - \rangle_\Lambda}^{\sstab} \simeq \bbP_\Bbbk^n.
\]
\end{proposition}

\begin{proof}
This is a part of~\cite{DomokosLenzing2002}*{Theorem~7.1}. One may also give a more direct proof, using a description of the semi-invariants for concealed-canonical algebras (the support of $\bh$ is a concealed-canonical algebra) which imply that $\bigoplus_{q \in \bbN} \SI [\mod_\Lambda (n \bh)]_{q \langle \bh, - \rangle_\Lambda}$ is the polynomial ring in $n + 1$ variables (see~\cite{Bobinski2012}*{Proposition~6.2}).
\end{proof}

\bibsection


\begin{biblist}

\bib{AssemSimsonSkowronski}{book}{
   author={Assem, I.},
   author={Simson, D.},
   author={Skowro{\'n}ski, A.},
   title={Elements of the Representation Theory of Associative Algebras. Vol. 1},
   series={London Math. Soc. Stud. Texts},
   volume={65},
   publisher={Cambridge Univ. Press },
   place={Cambridge},
   date={2006},
   pages={x+458},
}

\bib{AuslanderReitenSmalo}{book}{
   author={Auslander, M.},
   author={Reiten, I.},
   author={Smal{\o}, S. O.},
   title={Representation Theory of Artin Algebras},
   series={Cambridge Stud. Adv. Math.},
   volume={36},
   publisher={Cambridge Univ. Press},
   place={Cambridge},
   date={1997},
   pages={xiv+425},
}

\bib{BarotSchroer}{article}{
   author={Barot, M.},
   author={Schr{\"o}er, J.},
   title={Module varieties over canonical algebras},
   journal={J. Algebra},
   volume={246},
   date={2001},
   pages={175--192},
}

\bib{Bobinski2012}{article}{
   author={Bobi{\'n}ski, G.},
   title={Semi-invariants for concealed-canonical algebras},
   eprint={arXiv:1212.4013},
}

\bib{BobinskiSkowronski}{article}{
   author={Bobi{\'n}ski, G.},
   author={Skowro{\'n}ski, A.},
   title={Geometry of modules over tame quasi-tilted algebras},
   journal={Colloq. Math.},
   volume={79},
   date={1999},
   pages={85--118},
}

\bib{Chindris2009}{article}{
   author={Chindris, C.},
   title={Orbit semigroups and the representation type of quivers},
   journal={J. Pure Appl. Algebra},
   volume={213},
   date={2009},
   pages={1418--1429},
}

\bib{Chindris2011}{article}{
   author={Chindris, Calin},
   title={Geometric characterizations of the representation type of hereditary algebras and of canonical algebras},
   journal={Adv. Math.},
   volume={228},
   date={2011},
   pages={1405--1434},
}
		
\bib{Chindris2013}{article}{
   author={Chindris, C.},
   title={On the invariant theory for tame tilted algebras},
   journal={Algebra Number Theory},
   volume={7},
   date={2013},
   pages={193--214},
}

\bib{CarrollChindris}{article}{
   author={Carroll, A.},
   author={Chindris, C.},
   title={On the invariant theory for for acyclic gentle algebras},
   journal={Algebra Number Theory},
   volume={7},
   date={2013},
   pages={193--214},
}

\bib{CrawleyBoevey}{article}{
   author={Crawley-Boevey},
   title={On tame algebras and bocses},
   journal={Proc. London Math. Soc. (3)},
   volume={56},
   date={1988},
   number={3},
   pages={451--483},
}
		
\bib{Crawley-BoeveySchroer}{article}{
   author={Crawley-Boevey, W.},
   author={Schr{\"o}er, J.},
   title={Irreducible components of varieties of modules},
   journal={J. Reine Angew. Math.},
   volume={553},
   date={2002},
   pages={201--220},
}

\bib{DerksenWeyman2000}{article}{
   author={Derksen, H.},
   author={Weyman, J.},
   title={Semi-invariants of quivers and saturation for Littlewood-Richardson coefficients},
   journal={J. Amer. Math. Soc.},
   volume={13},
   date={2000},
  pages={467--479},
}

\bib{DerksenWeyman2002}{article}{
   author={Derksen, H.},
   author={Weyman, J.},
   title={Semi-invariants for quivers with relations},
   journal={J. Algebra},
   volume={258},
   date={2002},
   pages={216--227},
}

\bib{Domokos}{article}{
   author={Domokos, M.},
   title={Relative invariants for representations of finite dimensional algebras},
   journal={Manuscripta Math.},
   volume={108},
   date={2002},
   pages={123--133},
}

\bib{DomokosLenzing2002}{article}{
   author={Domokos, M.},
   author={Lenzing, H.},
   title={Moduli spaces for representations of concealed-canonical algebras},
   journal={J. Algebra},
   volume={251},
   date={2002},
   pages={371--394},
}

\bib{Drozd}{collection.article}{
   author={Drozd, Yu. A.},
   title={Tame and wild matrix problems},
   booktitle={Representation Theory, II},
   series={Lecture Notes in Math.},
   volume={832},
   publisher={Springer},
   place={Berlin},
   date={1980},
   pages={242--258},
}
		
\bib{Happel}{article}{
   author={Happel, D.},
   title={A characterization of hereditary categories with tilting object},
   journal={Invent. Math.},
   volume={144},
   date={2001},
   pages={381--398},
}

\bib{HappelReitenSmalo}{article}{
   author={Happel, D.},
   author={Reiten, I.},
   author={Smal{\o}, S. O.},
   title={Tilting in abelian categories and quasitilted algebras},
   journal={Mem. Amer. Math. Soc.},
   volume={120},
   date={1996},
   pages={viii+88},
}


\bib{Kerner}{article}{
   author={Kerner, O.},
   title={Tilting wild algebras},
   journal={J. London Math. Soc. (2)},
   volume={39},
   date={1989},
   pages={29--47},
}

\bib{King}{article}{
   author={King, A. D.},
   title={Moduli of representations of finite-dimensional algebras},
   journal={Quart. J. Math. Oxford Ser. (2)},
   volume={45},
   date={1994},
   pages={515--530},
}

\bib{LenzingSkowronski}{article}{
   author={Lenzing, H.},
   author={Skowro{\'n}ski, A.},
   title={Quasi-tilted algebras of canonical type},
   journal={Colloq. Math.},
   volume={71},
   date={1996},
   pages={161--181},
}

\bib{Meltzer}{article}{
   author={Meltzer, H.},
   title={Auslander-Reiten components for concealed-canonical algebras},
   journal={Colloq. Math.},
   volume={71},
   date={1996},
   pages={183--202},
}
		
\bib{delaPena}{article}{
   author={de la Pe{\~n}a, J. A.},
   title={On the dimension of the module-varieties of tame and wild algebras},
   journal={Comm. Algebra},
   volume={19},
   date={1991},
   pages={1795--1807},
}

\bib{Richmond}{article}{
   author={Richmond, N. J.},
   title={A stratification for varieties of modules},
   journal={Bull. London Math. Soc.},
   volume={33},
   date={2001},
   pages={565--577},
}

\bib{Ringel}{book}{
   author={Ringel, C. M.},
   title={Tame Algebras and Integral Quadratic Forms},
   series={Lecture Notes in Math.},
   volume={1099},
   publisher={Springer},
   place={Berlin},
   date={1984},
   pages={xiii+376},
}

\bib{Schofield}{article}{
   author={Schofield, A.},
   title={Semi-invariants of quivers},
   journal={J. London Math. Soc. (2)},
   volume={43},
   date={1991},
   number={3},
   pages={385--395},
}
		
\bib{SchofieldvandenBergh}{article}{
   author={Schofield, A.},
   author={van den Bergh, M.},
   title={Semi-invariants of quivers for arbitrary dimension vectors},
   journal={Indag. Math. (N.S.)},
   volume={12},
   date={2001},
   pages={125--138},
}

\bib{Skowronski}{article}{
   author={Skowro{\'n}ski, A.},
   title={Tame quasi-tilted algebras},
   journal={J. Algebra},
   volume={203},
   date={1998},
   pages={470--490},
}

\bib{SkowronskiWeyman}{article}{
   author={Skowro{\'n}ski, A.},
   author={Weyman, J.},
   title={The algebras of semi-invariants of quivers},
   journal={Transform. Groups},
   volume={5},
   date={2000},
   pages={361--402},
}
		
\bib{SkowronskiZwara}{article}{
   author={Skowro{\'n}ski, A.},
   author={Zwara, G.},
   title={Degenerations for indecomposable modules and tame algebras},
   journal={Ann. Sci. \'Ecole Norm. Sup. (4)},
   volume={31},
   date={1998},
   pages={153--180},
}
		
\end{biblist}
\end{document}